\documentclass[a4paper]{amsart}
\usepackage{dsfont}
\usepackage{amsmath,amsthm,amssymb,hyperref,mathrsfs}
\usepackage{aliascnt}
\usepackage{lmodern}
\usepackage[T1]{fontenc}
\usepackage[textsize=footnotesize]{todonotes}

\usepackage[noadjust,nocompress]{cite}

\usepackage{xcolor}
\definecolor{dblue}{rgb}{0,0,0.70}
\hypersetup{
	unicode=true,
	colorlinks=true,
	citecolor=dblue,
	linkcolor=dblue,
	anchorcolor=dblue
}

\makeatletter
\expandafter\g@addto@macro\csname th@plain\endcsname{%
		\thm@notefont{\bfseries}
	}%
\expandafter\g@addto@macro\csname th@remark\endcsname{%
		\thm@headfont{\bfseries}
	}%
\makeatother

\usepackage[nameinlink,capitalise,noabbrev]{cleveref}

\newtheorem{theorem}{Theorem}[section]

\newtheorem{lemma}[theorem]{Lemma}
\newtheorem{remark}[theorem]{Remark}

\newtheorem{corollary}[theorem]{Corollary}

\newtheorem{definition}[theorem]{Definition}
\newtheorem{question}[theorem]{Question}

\newtheorem*{lemma*}{Lemma}
\newtheorem*{theorem*}{Theorem}

\renewcommand{\restriction}{\mathbin\upharpoonright}

\newcommand{\axiom}[1]{\mathsf{#1}}
\newcommand{\ZFC}{\axiom{ZFC}}
\newcommand{\AC}{\axiom{AC}}

\newcommand{\DC}{\axiom{DC}}
\newcommand{\ZF}{\axiom{ZF}}

\newcommand{\GCH}{\axiom{GCH}}

\newcommand{\HS}{\axiom{HS}}

\DeclareMathOperator{\supp}{supp}

\DeclareMathOperator{\sym}{sym}

\DeclareMathOperator{\fix}{fix}

\DeclareMathOperator{\id}{id}
\DeclareMathOperator{\aut}{Aut}

\DeclareMathOperator{\Add}{Add}
\DeclareMathOperator{\crit}{crit}

\DeclareMathOperator{\Ult}{Ult}

\newcommand{\forces}{\mathrel{\Vdash}}

\newcommand{\PP}{\mathbb P}
\newcommand{\QQ}{\mathbb Q}
\newcommand{\RR}{\mathbb R}

\newcommand{\cC}{\mathcal C}

\newcommand{\cS}{\mathcal S}

\newcommand{\sF}{\mathscr F}
\newcommand{\sG}{\mathscr G}
\newcommand{\sH}{\mathscr H}

\newcommand{\1}{\mathds 1}

\newcommand{\tup}[1]{\langle#1\rangle}

\author{Asaf Karagila}
\author{Jiachen Yuan}
\email{karagila@math.huji.ac.il}
\urladdr{https://karagila.org}
\email{j.yuan@leeds.ac.uk}
\address{School of Mathematics,
    University of Leeds.
    Leeds, LS2~9JT, UK}
\thanks{The authors were supported by a UKRI Future Leaders Fellowship [MR/T021705/2].}

\date{30 July, 2026}

\title{Critical embeddings}

\begin{document}
\begin{abstract}
  Hayut and the first author isolated the notion of a \emph{critical cardinal} in \cite{KaragilaHayut:Critical}. In this work, we answer several questions raised in the original paper. We show that it is consistent for a critical cardinal not to have any ultrapower elementary embeddings, as well as that it is consistent that no target model is closed. We also prove that if $\kappa$ is a critical point by any ultrapower embedding, then it is the critical point of an ultrapower embedding by a normal measure. The paper concludes by presenting several open questions of interest in the study of critical cardinals.
\end{abstract}
%\begingroup
%\def\uppercasenonmath#1{} % this disables uppercasing title
%\let\MakeUppercase\relax % this disables uppercasing authors
\maketitle              % because ASL journals have no sense of decapitalising words where appropriate by themselves. Yes, I said it.
% \endgroup
\section{Introduction}
Recall that a cardinal $\kappa$ is a \textit{measurable cardinal} if it carries a $\kappa$-complete free ultrafilter, which we will refer to as a ``measure on $\kappa$''. In the context of $\ZFC$, this is equivalent to the existence of a non-trivial elementary embedding from $V_{\kappa+1}$,\footnote{We only care about elementary embeddings, and so from this point onwards ``embedding'' will always mean ``\emph{elementary embedding}''.} or even $V$ itself, into a transitive set or class, such that the least ordinal moved by the embedding is $\kappa$, in which case we say that $\kappa$ is the \emph{critical point} of the embedding.

In the context of $\ZF$, however, $\omega_1$ can be a measurable cardinal, as shown by Jech \cite{Jech:1968}, but it is not hard to show that $\omega_1$ can never be the critical point of an elementary embedding from $V_{\omega_1+1}$ into a transitive set. When $\kappa$ is the critical point of an embedding from $V_{\kappa+1}$ into a transitive set we say that it is a \emph{critical cardinal}. The basic theory of critical cardinals in $\ZF$ is very similar to the basic theory of measurable cardinals in $\ZFC$, and this was studied in \cite{KaragilaHayut:Critical}.

However, despite this, some questions remained open. For example, if $\kappa$ is a critical cardinal, assuming $\ZFC$, we can derive an ultrapower embedding by a normal measure, which will always be $\kappa^+$-closed.\footnote{That is, the target model is closed under sequences of any length below $\kappa^+$ itself.} Can we always do the same in $\ZF$? Are there always ultrapower embeddings, given a critical cardinal? In this work, we answer both of these in the negative.

\section{Preliminaries}
We follow the standard definitions of forcing (see \cite{Jech:ST2003} for a comprehensive exposition), where a notion of forcing is a preordered set, $\PP$, with a maximum $\1$. The elements of $\PP$ are called \emph{conditions}, and if $q\leq p$, we say that $q$ \emph{extends} $p$ or that it is a \emph{stronger} condition. We will use $\dot x$ to denote $\PP$-names, and given a collection of $\PP$-names in the ground model, $X$, we define the name $X^\bullet=\{\tup{\1,\dot x}\mid \dot x\in X\}$. This notation extends naturally to ordered pairs and sequences which are indexed by ground model sets. This allows us to neatly define the canonical names for ground model sets, $\check x=\{\check y\mid y\in x\}^\bullet$.

As we are working in $\ZF$, not all sets can necessarily be well-ordered. However, in the context of this work we will not discuss cardinals and cardinalities in general. It is therefore sufficient to remark that Greek letters, when denoting cardinals, will always refer to cardinals of well-orderable sets, which are just the corresponding initial ordinals.

\subsection{Symmetric extensions}
One of the features of forcing is that a generic extension of a model satisfying $\AC$ will always satisfy $\AC$. This is a problem if we want to prove consistency results involving the failure of the Axiom of Choice. For this purpose, the technique of symmetric extension was developed. Let $\PP$ be a notion of forcing. If $\pi\in\aut(\PP)$, then $\pi$ acts on the $\PP$-names, defined recursively \[\pi\dot x=\{\tup{\pi p,\pi\dot y}\mid\tup{p,\dot y}\in\dot x\}.\]
The following is known as \emph{The Symmetry Lemma}.
\begin{lemma}[{\cite[Lemma~14.37]{Jech:ST2003}}]
For any $p\in\PP$, $\pi\in\aut(\PP)$, $\dot x$, and formula $\varphi$, \[p\forces\varphi(\dot x)\iff\pi p\forces\varphi(\pi\dot x).\]
\end{lemma}

Let us fix a group $\sG\subseteq\aut(\PP)$. We say that $\sF$ is a \emph{filter of subgroups} on $\sG$ if it is a non-empty collection of subgroups of $\sG$ which is closed under supergroups and finite intersections. We say that $\sF$ is a \emph{normal filter} if whenever $H\in\sF$ and $\pi\in\sG$, $\pi H\pi^{-1}\in\sF$.\footnote{This terminology somewhat clashes with the notions of normal filters and ultrafilters, which is why we refer to the ultrafilters as ``measures''.}

We say that $\tup{\PP,\sG,\sF}$ is a \emph{symmetric system} if $\PP$ is a notion of forcing, $\sG$ is a subgroup of $\aut(\PP)$, and $\sF$ is a normal filter of subgroups on $\sG$. A $\PP$-name, $\dot x$, is \emph{$\sF$-symmetric} if $\sym_\sG(\dot x)=\{\pi\in\sG\mid\pi\dot x=\dot x\}\in\sF$. If this property holds hereditarily to all the names appearing in $\dot x$ and so on, we say that $\dot x$ is \emph{hereditarily $\sF$-symmetric}, and we denote the class of all hereditarily $\sF$-symmetric names by $\HS_\sF$. When the context is clear, we will omit the subscripts if possible.

\begin{theorem}[{\cite[Lemma~15.51]{Jech:ST2003}}]
  Let $G\subseteq\PP$ be a $V$-generic filter. The class $\HS_\sF^G=\{\dot x^G\mid\dot x\in\HS_\sF\}$ is a transitive class of $V[G]$, contains $V$, and satisfies $\ZF$.
\end{theorem}

We call such an intermediate model a \emph{symmetric extension} of $V$. We have a forcing relation, $\forces^\HS$, defined by relativising $\forces$ to the class $\HS_\sF$, which satisfies the Forcing Theorem,\footnote{Namely, something is true if and only if some condition in the generic forces it to be true.} as well as a version of The Symmetry Lemma when $\pi\in\sG$.

The study of symmetric extensions may be of interest to the readers. We suggest \cite{Usuba:LS,Usuba:GeologySym,KS:KWP} as starting points, as well as \cite{Karagila:2019} and \cite{KS:Sym}, for a generalised framework for iterations and extensions of the basic technique.

\subsection{Lifting theorem}
In the seminal \cite{KaragilaHayut:Critical}, a basic lifting theorem gives a basic condition for lifting an elementary embedding to symmetric extensions. Crucially, the embedding was not just lifted between the symmetric extensions, but it was also amenable to the symmetric extension. That is to say, the lifted embedding is itself symmetric. In other words, the basic lifting theorem allowed us to preserve the property ``$\kappa$ is a critical cardinal''. In \cite[Theorem~2.13]{KaragilaHayut:Small}, a small (but essential) improvement was made, by observation that the original proof provides a better theorem. In this paper we will need only a particular case of this improved theorem.
\begin{theorem}[Simplified lifting theorem]\label{thm:lifting-local}
  Suppose that $j\colon V\to M$ is an elementary embedding and $\tup{\PP,\sG,\sF}$ is a symmetric system such that the following conditions hold in $V$:
  \begin{enumerate}
  \item $j(\PP)\cong\PP\times\RR$ with $j(p)=\tup{p,\1_\RR}$, and there is an $M$-generic filter for $\RR$.
  \item $j(\sG)\cong\sG\times\sH$ with $j(\sigma)=\tup{\sigma,\id}$.
  \item For every $K\in j(\sF)$, $\{\sigma\in\sG\mid j(\sigma)\in K\}\in\sF$.
  \end{enumerate}
  Then we can lift $j$ amenably to the symmetric extensions of $V$ and $M$ by $\cS$ and $j(\cS)$ respectively.
\end{theorem}
We sketch the proof, and the complete proof can be found in \cite[\S4]{KaragilaHayut:Critical}.
\begin{proof}
  Suppose that the conditions hold. For $\dot x\in\HS_{j(\sF)}$ we define a partial interpretation function by $H$, the $M$-generic filter for $\RR$. Since $H\in V$, this is a well-defined function in $V$. The definition is recursive, \[\dot x^H=\{\tup{p,\dot y^H}\mid\tup{\tup{p,r},\dot y}\in\dot x, r\in H\},\] where we identify the conditions in $j(\PP)$ with $\PP\times\RR$.

  For any $\pi\in\sG$ and $V$-generic $G\subseteq\PP$ we have that $(\pi(\dot x^H))^G=(j(\pi)\dot x)^H$. And so, if $\dot x\in\HS_{j(\sF)}$, then $\dot x^H\in\HS_\sF$. In particular, if $\dot a\in\HS_{\sF}$, then $j(\dot a)^H\in\HS_\sF$ as well.

  As a consequence, $\{\tup{\dot x,j(\dot x)^H}^\bullet\mid\dot x\in\HS\}^\bullet$ is stable under the action of $\sG$, and provides us with the lifted embedding.
\end{proof}
Note that assuming the second condition, the third condition is equivalent to $j``(K\cap\sG)\subseteq K$, where $K\cap\sG$ is taken as the projection onto the $\sG$-coordinate.
\subsection{Choice in measure}
\begin{definition}
  Let $U$ be a family of subsets of $I$. We say that $U$-$\AC_I$ holds if whenever $\{A_i\mid i\in I\}$ is a family of non-empty sets, there is some $J\in U$ such that $\{A_i\mid i\in J\}$ admits a choice function. Namely, $\prod_{i\in J}A_i\neq\varnothing$.
\end{definition}
Normally we will require $I$ to be some $\kappa$ and $U$ to be a measure on $\kappa$. In the following theorem, note that we do not require that $\kappa$ is measurable or that $U$ is $\sigma$-complete.

\begin{theorem}[Spector \cite{Spector:1988}]\label{thm:Spector}
  Suppose that $U$ is an ultrafilter on $\kappa$. The following are equivalent:
  \begin{enumerate}
  \item $j\colon V\to\Ult(V,U)$ is elementary.
  \item $\Ult(V,U)$ satisfies the Axiom of Extensionality.
  \item $\Ult(V,U)\models\forall x(x\notin[f]_U)$ if and only if $[f]_U=[c_\varnothing]_U$.
  \item $U$-$\AC_\kappa$ holds.
  \end{enumerate}
\end{theorem}
The original theorem is formulated with a few more equivalences, but these are arguably reformulations of (1). One noteworthy remark is that we are not even assuming that $M$ is well-founded, so the theorem also holds for ill-founded ultrapowers. Note that the ultrapower embedding is well-defined without any choice. We will call these ``ultrapower maps'' to suggest that they might not be elementary.

We will use this theorem to show that amongst other things, $U$-$\AC_\kappa$ does not even imply $\AC_\omega$. This is quite counterintuitive at first, but from the perspective of a Levy--Solovay phenomenon this is in fact quite natural. It means that we have ``wiggle room'' to introduce small failures of choice, as long as we preserve choice modulo $U$.

As a consequence of Jech's work in \cite{Jech:1968} (see also Theorem~2.7 in \cite{KaragilaHayut:Small}), we have that if $\kappa$ is measurable and $\cS\in V_\kappa$ is a symmetric system, then any measure on $\kappa$ extends uniquely to a measure in the symmetric extension. Additionally, the lifting theorem shows that the elementary embeddings lift as well to the symmetric extension, and indeed the ultrapower embeddings remain ultrapower embeddings. Since we can violate $\AC_\omega$ by adding countably many Cohen reals,\footnote{Equivalently, a single Cohen real.} with the most famous construction being Cohen's first model (see Chapter~5 in \cite{Jech:AC}), we immediately get the following corollary.
\begin{corollary}
  If $\kappa$ is a measurable cardinal and $U$ is a measure on $\kappa$, then $U$-$\AC_\kappa$ does not imply $\AC_\omega$.
\end{corollary}

Recall that if $U_0$ and $U_1$ are two ultrafilters on $\kappa$, then $U_0\leq_{RK} U_1$ if and only if there is some $r\colon\kappa\to\kappa$ such that $A\in U_0\iff r^{-1}(A)\in U_1$.

\begin{theorem}
  Suppose that $U_0\leq_{RK}U_1$ are two ultrafilters on $\kappa$. Then $U_1$-$\AC_\kappa$ implies $U_0$-$\AC_\kappa$.
\end{theorem}
\begin{proof}
  Let $r\colon\kappa\to\kappa$ be such that $A\in U_0\iff r^{-1}(A)\in U_1$. Suppose that $f$ is a function with domain $\kappa$ and $f(\alpha)\neq\varnothing$ for all $\alpha$. We want to show there is some $U_0$-positive set on which $f$ admits a choice function. Let $F=f\circ r$ and let $H$ be a function such that $B=\{\beta<\kappa\mid H(\beta)\in F(\beta)\}\in U_1$, which exists by $U_1$-$\AC_\kappa$.

  For $\alpha<\kappa$, let $\alpha_*=\min\{\beta\in B\mid r(\beta)=\alpha\}$ when it is defined. Note that if $\beta\in B$, then $r(\beta)=r(r(\beta)_*)$, so in particular \[F(\beta)=f(r(\beta))=f(r(r(\beta_*)))=F(r(\beta)_*).\] Now, letting $h(\alpha)=H(\alpha_*)$, we claim that $C=\{\alpha<\kappa\mid h(\alpha)\in f(\alpha)\}\in U_0$. It is enough to show that $r``B\subseteq C$, as this would imply that $B\subseteq r^{-1}(C)$, so $r^{-1}(C)\in U_1$, which implies $C\in U_0$ as wanted.

  Suppose that $\beta\in B$ and let $\alpha=r(\beta)$. By the above argument, $\alpha_*\in B$. And so, \[h(\alpha)=H(\alpha_*)=H(r(\beta)_*)\in F(r(\beta)_*)=F(\beta)=f(r(\beta))=f(\alpha).\qedhere\]
\end{proof}
\begin{corollary}
  Suppose that $\kappa$ is a measurable cardinal, $U$ is a measure on $\kappa$, and $U$-$\AC_\kappa$ holds. Let $D$ be the normal measure derived from $j_U$. Then $D$-$\AC_\kappa$ holds.
\end{corollary}
\begin{proof}
  Let $r\colon\kappa\to\kappa$ be a function representing $\kappa$ in $\Ult(V,U)$. Then $r$ witnesses that $D\leq_{RK}U$, so by the above theorem $D$-$\AC_\kappa$ holds.
\end{proof}
\begin{corollary}
  Suppose that $U$ is a measure on $\kappa$ such that $j_U\colon V\to\Ult(V,U)$ is elementary, and let $D$ be the normal measure derived from $j_U$. Then the ultrapower embedding $j_D\colon V\to\Ult(V,D)$ is elementary and $j_U$ factors through it.
\end{corollary}
\begin{proof}
  Let $r\colon\kappa\to\kappa$ be a function representing $\kappa$ in $\Ult(V,U)$. By the previous corollary and \autoref{thm:Spector}, $D$-$\AC_\kappa$ holds. By \autoref{thm:Spector}, $j_D\colon V\to\Ult(V,D)$ is elementary. Define $k\colon\Ult(V,D)\to\Ult(V,U)$ by $k([f]_D)=[f\circ r]_U$.

  To see that $k$ is indeed elementary, note that
  \begin{align*}\Ult(V,D)\models\varphi([f]_D)&\iff\{\alpha<\kappa\mid\varphi(f(\alpha))\}\in D\\
                                              &\iff\{\alpha<\kappa\mid\varphi(f(r(\alpha)))\}\in U\\
                                              &\iff\Ult(V,U)\models\varphi([f\circ r]_U)\\
                                              &\iff\Ult(V,U)\models\varphi(k([f]_D)).
  \end{align*}
  A standard verification shows that $k([f]_D)=j_U(f)(\kappa)$, and so $k$ is indeed the factor embedding.
\end{proof}

\section{Main theorems}\label{sec:main}
Both our theorems are proved with the same construction, and unless stated otherwise, ground models will satisfy $\ZFC+\GCH$.
\subsection{No ultrapower maps are elementary}
\begin{theorem}\label{thm:no-up}
  If $\kappa$ is a measurable cardinal, then there is a symmetric extension in which $\kappa$ remains a critical cardinal, but no ultrapower map is elementary. In other words, $U$-$\AC_\kappa$ fails for any measure $U$ on $\kappa$.
\end{theorem}
\begin{proof}
  In $V$, let $A\subseteq\kappa$ be the set of inaccessible non-Mahlo cardinals. We will add to each $\alpha\in A$ some Cohen subsets collected into $C_\alpha$, so that there is no way to choose one such subset for each $\alpha$ for any unbounded subset of $A$. This, in turn, will mean that the family $\{C_\alpha\mid\alpha<\kappa\}$ does not have any $\kappa$-sized subfamily admitting a choice function, so $U$-$\AC_\kappa$ fails for any potential measure $U$. Therefore, by Spector's theorem, the ultrapower map is not elementary.

  For each $\alpha\in A$, we define a symmetric system $\cS_\alpha=\tup{\QQ_\alpha,\sG_\alpha,\sF_\alpha}$. We let $\QQ_\alpha=\Add(\alpha,\alpha)$ be the forcing with conditions $p\colon\alpha\times\alpha\to 2$ such that $|p|<\alpha$, $\sG_\alpha$ be the group of all permutations of $\alpha$ moving fewer than $\alpha$ many points with the the action $\pi p(\pi\beta,\gamma)=p(\beta,\gamma)$, and $\sF_\alpha$ be the filter generated by fixing pointwise finite subsets of $\alpha$.

  Our symmetric system is $\cS=\tup{\PP,\sG,\sF}$, where $\PP$ is the Easton support product of the $\QQ_\alpha$, $\sG$ is the Easton support product of $\sG_\alpha$ acting pointwise on $\PP$, and $\sF$ is generated by groups of the form $\vec H=\tup{H_\alpha\mid\alpha\in A}$, where $H_\alpha\in\sF_\alpha$ and only finitely many coordinates satisfy $H_\alpha\neq\sG_\alpha$.\footnote{One might wonder why we are using the Easton support product for $\PP$ and $\sG$, but not $\sF$. The truth is that it would not make a difference. If one understands symmetric systems by their tenacious conditions (i.e., those which are fixed in place by some large group), then this symmetric system is the same as the finite support product, and indeed the finite support conditions and permutations at each $\alpha$ as well. (See \cite[Appendix~A]{Karagila:2019} and \cite{KS:Sym} for full details and context.) However, using Easton support makes it easier to justify the existence of an $M$-generic filter.}

  We claim that $\cS$ satisfies the conditions of \autoref{thm:lifting-local}, for any $j\colon V\to M$ for which $\crit(j)=\kappa$ and $M^\kappa\subseteq M$, e.g., any ultrapower embedding. Let us verify each condition.
  \begin{enumerate}
  \item Note that $j(\PP)\cong\PP\times\RR$ with $j(p)=\tup{p,\1_\RR}$, where $\RR$ is the Easton support product of $\Add(\alpha,\alpha)$ for $\alpha\in j(A)\setminus\kappa$.
    Since $\kappa$ is Mahlo in $M$ and $\kappa\notin j(A)$, $M\models``\RR$ is $\kappa^+$-closed''. Therefore, we can construct in $V$ an $M$-generic filter for $\RR$.
  \item Note that $j(\sG)\cong\sG\times\sH$, where $\sH$ is the Easton support product of the $\sG_\alpha$, now permutations of $\alpha\in j(A)\setminus\kappa$. Moreover, since if $\vec\pi\in\sG$ is any permutation, it has an Easton support and so it is bounded below $\kappa$. So $j(\vec\pi)$ has same support. Therefore, $j(\vec\pi)=\tup{\vec\pi,\vec\id}$, which fulfils (2) in our simplified lifting theorem.
  \item Note that $j(\sF)$ is generated by the finite support product of the filters, so with the previous condition, if $\vec H\in j(\sF)$ is a generator of the filter, then it has a finite support, and so $\vec H\restriction\kappa\in\sF$. Therefore, $j``\vec H\restriction\kappa\subseteq\vec H$.
  \end{enumerate}

  We let $W$ and $N$ denote the symmetric extensions of $V$ and $M$ respectively defined from $G$ and $j(G)$, the generic filters. As the conditions for \autoref{thm:lifting-local} hold, $j$ lifts to $j\colon W\to N$ in an amenable manner. Let us see that no embedding is an ultrapower embedding, as we suggested earlier.

  For a fixed $\alpha\in A$, let $\dot c_{\alpha,\xi}$ denote the $\QQ_\alpha$-name for the $\xi$th Cohen subset. By the way we defined $\cS_\alpha$, if $\pi\in\sG_\alpha$, then $\pi\dot c_{\alpha,\xi}=\dot c_{\alpha,\pi\xi}$. Therefore each $\dot c_{\alpha,\xi}\in\HS_{\sF_\alpha}$, and so $\dot C_\alpha=\{\dot c_{\alpha,\xi}\mid\xi<\alpha\}^\bullet\in\HS_{\sF_\alpha}$. Going back to the context of $\PP$, it is not hard to see that the name $\dot\cC=\tup{\dot C_\alpha\mid\alpha\in A}^\bullet$ is in $\HS$, since the permutations in $\sG$ must preserve it pointwise.

  We argue that if $\dot B\in\HS$ and $p\forces^\HS\sup\dot B=\check\kappa$, then $p\forces^\HS\prod_{\alpha\in\dot B}\dot C_\alpha=\check\varnothing$. To see this, if this is not the case, then there is some $\dot f\in\HS$ and $q\leq p$ such that $q$ forces $\dot f$ to be in that product. Let $E$ be a finite set which is a support for $\dot f$, namely, if $\pi\restriction E=\vec\id$, then $\pi\dot f=\dot f$. Let $\alpha>\sup E$ be large enough such that $q\forces\dot f(\check\alpha)=\dot c_{\alpha,\xi}$ for some $\xi<\alpha$, if no such $\alpha$ and $\xi$ exist, we can extend $q$ as necessary. Let $\zeta<\alpha$ be different from $\xi$ such that $\zeta\notin\supp q(\alpha)\in\Add(\alpha,\alpha)$. Let $\pi_\alpha$ be the $2$-cycle $(\xi\ \zeta)$ and let $\pi_\beta=\id$ for $\beta\neq\alpha$. Then $\vec\pi\in\sG$, and moreover, $\vec\pi\dot f=\dot f$. It is not hard to check that $\vec\pi q$ is compatible with $q$, which is impossible, since this implies that $q\cup\vec\pi q\forces^\HS\dot f(\check\alpha)=\dot c_{\alpha,\xi}\neq\dot c_{\alpha,\zeta}=\dot f(\check\alpha)$.

  Therefore in $W$, the family $\cC$ does not admit any subset of size $\kappa$ which has a choice function. In particular, there is no measure on $\kappa$ for which $U$-$\AC_\kappa$ holds, and so by \autoref{thm:Spector} there are no ultrapower embeddings with critical point $\kappa$.
\end{proof}
\begin{remark}
  Suppose $\lambda<\kappa$ is regular and we replace ``finite support'' by $\lambda$-support, while restricting $\PP$ to the cardinals above $\lambda$. Then $\PP$ is $\lambda$-closed, $\sF$ will be $\lambda$-complete, so by \cite{Karagila:DC}, $W\models\DC_{<\lambda}$. This shows that $\AC_\omega$ is consistent with the negation of $U$-$\AC_\kappa$ (assuming the consistency of a measurable cardinal).
\end{remark}
\subsection{No closed target models}
\begin{theorem}\label{thm:no-closed-targets}
  If $\kappa$ is a measurable cardinal, then there is a symmetric extension in which $\kappa$ remains a critical cardinal, but no embedding with critical point $\kappa$ has a countably closed target model.
\end{theorem}
\begin{proof}
  The construction is the same as in \autoref{thm:no-up}, so we will retain the notation from its proof. Suppose that $j\colon V_{\kappa+1}^W\to N$ is an elementary embedding with critical point $\kappa$.

  Since $j(\cC)\in W$ is a sequence of length $j(\kappa)$, and $j(\kappa)>\kappa^+$, we have that for some $\alpha>\kappa^+$, $j(\cC)_\alpha$ is not empty. In $V$, let $\dot C_\alpha$ be a name for $j(\cC)_\alpha$ which is in $\HS$. Since $V\models|\dot C_\alpha|\geq\kappa^+$, we can find a subset of $\dot C_\alpha$ which is of size $\kappa^+$ and has the same stabiliser. But this means that the name for a well-order of this subset will be stabilised by the same finite set, and therefore $W\models\kappa^+\leq|C_\alpha|$. However, standard arguments show that in $N$, $j(\cC)_\alpha$ is Dedekind-finite, but not so in $W$. So in particular, $N$ is not countably closed in $W$.
\end{proof}

As with the case of the ultrapower maps being non-elementary, this construction lends itself to easy generalisations when replacing ``finite support'' by $\lambda$-support, for some regular $\lambda<\kappa$, to obtain that no target model is $\lambda$-closed while $\DC_{<\lambda}$ holds. But even more broadly, if $j(\PP)=\PP\times\RR$ (with $\RR$ non-trivial) and the conditions for \autoref{thm:lifting-local} hold, then if the symmetric extension given by $j(\cS)$ is not $\lambda$-closed in $M[H]$, the lifted embeddings will not be $\lambda$-closed either. The reason, of course, is simple: $H$ is in $W$, so $M[H]$ is a class of $W$, and the failure of closure is accessible to $W$. This is easily the case when adding Dedekind-finite sets, as above.

Interestingly, however, the requirement that $\RR$ is non-trivial is essential. In cases such as Cohen's first model, where $j(\PP)=\PP$, and we simply rely on the Levy--Solovay phenomenon to lift the embedding, the embeddings remain closed: Any Dedekind-finite set in $N$ is going to be Dedekind-finite in $W$, and so the counterexamples to the closure of $N$ are simply not inside $W$.
\section{Questions}
\begin{question}
  Can a critical cardinal admit an ultrapower embedding, but not a closed one?
\end{question}
To try and attack the above question, we can consider the following local symmetric systems: For an inaccessible cardinal $\alpha<\kappa$, consider $\Add(\alpha^+,\omega\times\alpha^+)$. This defines a natural partition of the new subsets into countably many sets. Take the wreath product $\{\id\}\wr S_{<\alpha^+}$ as our automorphism group, and the filter generated by $\fix(n\times\alpha^+)$, fixing pointwise the first $n$ cells of the partition. This system adds a countable sequence of well-orderable sets which does not admit a choice function. We can now take the product of these symmetric systems.
\begin{enumerate}
\item  In the full support product of these symmetric systems, the conditions of the lifting theorems (neither the simplified one, nor \cite[Theorem~2.13]{KaragilaHayut:Small}) no longer hold: While $j(\PP)\cong\PP\times\RR$, it is not true that $j(p)=\tup{p,\1_\RR}$, for example. Even if an embedding lifts, it is not clear what a counterexample to the closure of the target model will look like, nor if $U$-$\AC_\kappa$ holds for any measure.
\item If we take an Easton support product (of all three components, including the filters), we can find the needed $M$-generic, as we did in the theorems presented in this paper. However, it is not at all clear what sort of argument can be used to prove that $U$-$\AC_\kappa$ holds. Indeed, it is not clear that it does hold in the extension.
\end{enumerate}
It might be the case, however, that stronger assumptions (e.g., $o(\kappa)>1$) are needed to allow for a more refined symmetric system. This lends itself to the following question.

\begin{question}
  Is the consistency strength of ``$\kappa$ admits ultrapower embeddings, but no closed ultrapower embeddings'' higher than a single measurable cardinal?
\end{question}

Independently of the two questions above, the following are two important questions.
\begin{question}
  Is there a lifting criterion which provides us with an easy-to-use control over the closure of the target models?
\end{question}
\begin{question}
  When taking a symmetric extension over a model with a measurable cardinal, are there new embeddings which are not generated by ground model embeddings? Can we find a condition which allows us to control these?
\end{question}
If there is a reasonably usable criterion which allows us to control the closure of the lifted embeddings, one can conceivably concoct, starting from a single supercompact cardinal, a class-sized product or iteration in which enough supercompact embeddings lift and retain their closure. This will provide an answer to Question~3.11 in \cite{KaragilaHayut:Critical}, asking for a model where a supercompact cardinal exists and the Axiom of Choice fails in a non-trivial way (that is, not a Levy--Solovay phenomenon).

\subsection*{Acknowledgements}
The authors would like to thank the anonymous referee for suggesting improvements to the presentation of the paper.
\newpage
\providecommand{\bysame}{\leavevmode\hbox to3em{\hrulefill}\thinspace}
\providecommand{\MR}{\relax\ifhmode\unskip\space\fi MR }
\providecommand{\MRhref}[2]{%
  \href{http://www.ams.org/mathscinet-getitem?mr=#1}{#2}
}
\providecommand{\href}[2]{#2}

\end{document}